\documentclass[11pt, reqno, draft]{amsart}
\usepackage{mathrsfs, amsfonts, amssymb, latexsym, amsmath}
\usepackage{graphicx,epsf,amsbsy,bm}

\theoremstyle{plain}

\newtheorem{proposition}{Proposition}
\newtheorem{corollary}{Corollary}[proposition]
\newtheorem{remark}{Remark}[section]

\title[Feynman integral in $\mathbb R^1\oplus\mathbb R^m$ ...]
{Feynman integral in $\bm{\mathbb R^1\oplus\mathbb R^m}$ and complex expansion of $\bm{{}_2F_1}$}

\author{Mykola A. Shpot}
\address{(M. A. Shpot) Institute for Condensed Matter Physics, 79011 Lviv, Ukraine}
\email{shpot.mykola@gmail.com}

\author{Tibor K. Pog\'any}
\address{(T. K. Pog\'any) Faculty of Maritime Studies, University of Rijeka, 51000 Rijeka, Croatia \and
Institute of Applied Mathematics, \'Obuda University, 1034 Budapest, Hungary}
\email{poganj@pfri.hr}

\newcommand{\e}{{\rm e}}

\newcommand{\R}{{\rm Re}\,}
\newcommand{\E}{\epsilon}
\newcommand {\lr}[1]{\left({#1}\right)}

\newcommand {\lrB}[1]{\Big({#1}\Big)}
\newcommand {\lrs}[1]{\left[{#1}\right]}
\newcommand {\lrS}[1]{\Big[{#1}\Big]}

\def\d#1{{#1\kern-0.4em\char"16\kern-0.1em}}
\def\D#1{{\raise0.2ex\hbox{-}\kern-0.4em #1}}
\dedicatory{\it To Professor Hari Mohan Srivastava on the occasion of his 75th birthday anniversary}

\begin{document}
\maketitle

\begin{abstract}
Closed form expressions are proposed for the Feynman integral
   \begin{equation} \label{Ipq}
      I_{D, m}(p,q) = \int\frac{{\rm d}^m y}{(2\pi)^m}\int\frac{{\rm d}^D x}{(2\pi)^D}
                         \frac1{(\bm x-\bm p /2)^2+(\bm y-\bm q /2)^4}\,
                         \frac1{(\bm x+\bm p /2)^2+(\bm y+\bm q /2)^4}
    \end{equation}
over $d=D+m$ dimensional space with $(\bm x,\bm y),\, (\bm p,\bm q) \in \mathbb R^D \oplus \mathbb R^m$, in the special
case $D=1$. We show that $I_{1,m}(p,q)$ can be expressed in different forms involving real and imaginary parts of the complex variable
Gauss hypergeometric function $_2F_1$, as well as generalized hypergeometric $_2F_2$ and $_3F_2$, Horn $H_4$ and  Appell $F_2$
functions. Several interesting relations are derived between the real and imaginary parts of $_2F_1$ and the function $H_4$.

\vspace{3mm}

\vspace{1mm}
\end{abstract}

\allowdisplaybreaks

\section{Introduction}\label{Intro}

We are interested in calculation of the integral \eqref{Ipq} occurring in the field theoretical treatment of $m$ axial Lifshitz points
in strongly anisotropic $d$ dimensional systems (see e.g. \cite{SPD05,RDS11,SP12} and references therein).

The integral \eqref{Ipq} is over a $d$ dimensional space split into two complementary Euclidean subspaces $\mathbb R^D$ and
$\mathbb R^m$. The co-dimensions $D$ and $m$ are related via $D=d-m$. The integral \eqref{Ipq} converges inside of a sector of the
$(m,d)$ plane constrained by the border lines $d^*(m)=4+m/2$, $d_*(m)=2+m/2$, $m=0$, and $d=m$, see \cite[Fig. 1]{SPD05}. The function
$I_{1,m}(p,q)$ can be extended outside the convergence region of the integral \eqref{Ipq} by means of analytical continuations.

The deviations from the "critical" lines $d^*(m)$ and $d_*(m)$ are defined as $\varepsilon = 4+m/2-d$ and $\epsilon = d-2-m/2$;
hence $\varepsilon = 2-\epsilon$. When $d=d^*(m)$ the integral \eqref{Ipq} diverges at large $\bm x\in\mathbb R^D$ and
$\bm y\in\mathbb R^m$, while at $d=d_*(m)$ it diverges for small values of $\bm x$ and $\bm y$. Close to marginal dimensions
$d^*(m)$ and $d_*(m)$, these divergences manifest themselves as simple poles $\propto 1/\varepsilon$ and $\propto 1/\E$, respectively.

Due to the full rotational symmetry in each of the subspaces $\mathbb R^D$ and $\mathbb R^m$,
the integral $I_{D, m}(p,q) \equiv I(p,q)$ depends on the absolute values $p$, $q$ of the
vectors $\bm p\in\mathbb R^D$, $\bm q\in\mathbb R^m$. Moreover, it is a generalized homogeneous function. This implies
that a power of each argument, $p^{-\varepsilon} = p^{-2+\E}$ or $q^{-2\varepsilon}=q^{-4+2\E}$, can be scaled out and the
function $I(p,q)$ can be represented in two equivalent scaling forms,
   \begin{equation}\label{SCI}
     I(p,q) = p^{-2+\E}\,I\Big(1,\frac q{\sqrt p}\Big) \quad {\rm or} \quad
     I(p,q) = q^{-4+2\E}\,I\Big(\frac p{q^2},1\Big)\,.
   \end{equation}
The special cases of $I(p,q)$ at zero arguments are homogeneous in the sense that
   \begin{equation} \label{S00}
      I(p,0) = p^{-2+\E}\,I(1,0) \quad {\rm and} \quad I(0,q) = q^{-4+2\E}\,I(0,1).
   \end{equation}
Thus the constants $I(1,0)$ and $I(0,1)$ give the corresponding asymptotic behaviors of $I(p,q)$ at
$p\to\infty$ and $q\to\infty$ when the remaining second variable is finite.

For several pairs of $m$ and $D$, which represent certain separate \emph{points} in the $(m,d)$ plane,
$I_{D, m}(p,q)$ has been calculated explicitly \cite{SPD05}:
   \begin{align}\nonumber 
      I_{2,1}(1,q) &= \frac{\sqrt[4]{w}}{4\sqrt 2}\,_2F_1\Big(\frac{1}{4},\frac{1}{4};1;w\Big)\,,
			                \quad w = \frac4{(1+q^4)^2(4+q^4)}
                   &&(\varepsilon=3/2, \E=1/2), \\\nonumber
      I_{3,1}(1,q) &= {1\over 8\pi\sqrt 2} \sqrt{\sqrt{4+q^4}-q^2} = {q^{-1}\over 4\pi\sqrt 2}{1\over\sqrt{\sqrt{1+4/q^4}+1}}
                   &&(\varepsilon=1/2, \E=3/2), \\\label{I41}
      I_{1,4}(1,q) &= \frac1{32\pi^2}\,\bigg[ \arctan\frac{2}{q^2(3+q^4)}+\frac1{q^2}\,\ln{1+q^4\over 1+ q^4/4}\bigg]
                   &&(\varepsilon=\E=1)\, .
   \end{align}
The explicit dependence of $I_{D, m}(p,q)$ on $p$ can be reconstructed here by using the homogeneity property \eqref{SCI}.

The explicit expressions for $I(1,0)$ and $I(0,1)$ for arbitrary $D$ and $m$ have been derived in \cite[Appendices B.1 and B.2]{SPD05}.
In the special case $D=1$ they reduce to
   \begin{equation} \label{I10}
      I_{1,m}(1,0) = C_1\,\frac{2^{-2+\epsilon}}{1-\epsilon}\,\cos\frac{\pi\epsilon}{2}\,, \qquad
      I_{1,m}(0,1) = \frac{C_1}{2(1-\epsilon)}  \Big[ \epsilon-\frac{\sqrt\pi\,\epsilon\, \Gamma(\epsilon)}
                     {\Gamma(-\frac12+\epsilon)}\Big],
   \end{equation}
where $C_1 = C_1(\epsilon)=16^{-\epsilon}\pi^{-1-\epsilon}\Gamma(2-\epsilon)/\epsilon$.

In writing \eqref{I10} we use a convenient parametrization in terms of $\epsilon=m/2-1$ which follows from  $\epsilon=D+m/2-2$ at $D=1$.
Dividing the original integral \eqref{Ipq} by $C_1$ we define the reduced function $\hat I_{1,m}(p,q):= C_1^{-1}\,I_{1,m}(p,q)$.
The constant $C_1$ contains the both poles in $\epsilon$ and in $\varepsilon=2-\epsilon$. Thus, the function
$\hat I_{1,m}(p,q)$ is analytic in both of these variables.

In the sequel we shall calculate the integral \eqref{Ipq} along the \emph{line} in the $(m,d)$ plane given by $D=1$ and $m$ varying
in the range from $2$ to $6$. We shall encounter generalized hypergeometric functions ${}_rF_s$, as well as Horn $H_4$ and Appell  hypergeometric functions of two variables. Thus we need their power series definitions.

The generalized hypergeometric function ${}_rF_s(z)$ with $r$ numerator and $s$ deno\-minator parameters is defined by the series
   \begin{equation} \label{A4}
      {}_rF_s \Big( \begin{array}{c} a_1,\cdots,a_r \\ b_1, \cdots, b_s \end{array} \Big| \,z \Big)
            = \sum_{n \geq 0} \dfrac{ \prod\limits_{j=1}^r(a_j)_{n}}{\prod\limits_{k=1}^s(b_k)_{n}} \; \dfrac{z^n}{n!}\,,
   \end{equation}
where $a_j \in \mathbb C; j=\overline{1,r}$ and $b_k \in \mathbb C \setminus \mathbb Z_0^-; k= \overline{1,s}$. The $(\lambda)_{\mu}$
is  the Pochhammer symbol (or the {\it shifted factorial}) defined, in terms of Euler's Gamma function, by
   \[ (\lambda)_{\mu}:=\dfrac{\Gamma(\lambda+\mu)}{\Gamma(\lambda)}
                      =\begin{cases}
                           1 & \mu=0;\, \lambda \in \mathbb C \setminus \{0\}\\
                           \lambda(\lambda+1) \cdots (\lambda + n-1) & \mu=n \in \mathbb N;\; \lambda \in \mathbb C
                         \end{cases}\,,\]
and, by convention, $(0)_0=1$.

In the case $r=s+1$ the series \eqref{A4} converges for $|z| < 1$ and diverges for $|z| > 1$. On the unit circle $|z|=1$,
the sufficient condition for convergence of \eqref{A4} is $\Re\,\big(\sum_{j=1}^s b_j - \sum_{j=1}^r a_j\big)>0$.
Analytic continuation can be employed for larger $z$ values. When $r \leq s$ the series \eqref{A4} converges for all finite
values of $z$. 

The Appell's hypergeometric functions of two variables involved in the following are defined by double power series expansions
\cite{Appell26,Bailey,Slater,SriMan}
   \begin{align}\label{AF1}
      &F_1(a, b, b'; c; x, y) = \sum_{k\geq 0} \sum_{n\geq 0} \frac{(a)_{k+n} (b)_k (b')_n}{(c)_{k+n}}\,
                                \frac{x^k}{k!}\, \frac{y^n}{n!}, \qquad |x|<1,\,|y|<1, \\\label{AF2}
      &F_2(a,b,b';c,c';x,y)   = \sum_{k\geq 0} \sum_{n\geq 0} \frac{(a)_{k+n} (b)_k (b')_n}{(c)_k(c)_n}\,
                                \frac{x^k}{k!}\, \frac{y^n}{n!}, \qquad |x|+|y|<1, \\\label{AF4}
      &F_4(a,b;c,c';x,y)      = \sum_{k\geq 0} \sum_{n\geq 0} \frac{(a)_{k+n}(b)_{k+n}}{(c)_k(c')_n}\frac{x^k}{k!}\frac{y^n}{n!},
			                          \qquad \, \, \sqrt{|x|}+\sqrt{|y|}<1\,.			
   \end{align}
Finally, the Horn function $H_4$ (see e. g. \cite{SriMan, ErdT, SriKar} and their references) is defined by
   \begin{equation}\label{H4D}
      H_4(\alpha,\beta;\gamma,\delta;x,y) = \sum_{k \geq 0}\sum_{n\ge0}
                                            \frac{(\alpha)_{2k+n}(\beta)_n}{(\gamma)_k(\delta)_n}\,
                                            \frac{x^k}{k!}\,\frac{y^n}{n!}, \qquad 2\sqrt{|x|} + |y|<1.
   \end{equation}
Outside their convergence domains the functions \eqref{AF1}-\eqref{H4D} are defined by analy\-tical continuation.

\section{The inner integral}

In \eqref{Ipq}, the inner $D$ dimensional integral over $\bm x$ has the form
   \begin{equation} \label{IP01}
      \hspace{-1mm}  J_D(p;\kappa_1,\kappa_2){=}
      \int_{\bm x}^{(D)}\!\!\!\!\frac1{(\bm x{-}\bm p /2)^2{+}\kappa_1^2}\,\frac1{(\bm x{+}\bm p/2)^2{+}\kappa_2^2}=
      \int_{\bm x}^{(D)}\!\!\!\!\frac1{x^2{+}\kappa_1^2}\,\frac1{(\bm x{+}\bm p)^2{+}\kappa_2^2}.
   \end{equation}
This is the standard one-loop two-point Feynman integral with arbitrary "masses" $\kappa_1$ and $\kappa_2$ and external momentum
$\bm p$.\footnote{Here and throughout the paper we use an obvious short-hand notation
$(2\pi)^{-\Omega}\int{\rm d}^\Omega z \equiv \int_{\bm z}^{(\Omega)}$, $\Omega \in \{D, m\}$.}
As a function suitable for subsequent integration over $\bm y$ in \eqref{Ipq}, the integral \eqref{IP01} has been calculated in
\cite{Sh07}:
   \begin{equation} \label{Iff1}
      J_D(p;\kappa_1,\kappa_2)=\gamma_D\Big(\frac{\kappa_1+\kappa_2}{2}\Big)^{D-4}
      F_1\Big(2-\frac D2;1,\frac32-\frac D2;\frac32;
      \frac{-p^2}{(\kappa_1+\kappa_2)^2},\frac{(\kappa_2-\kappa_1)^2}{(\kappa_1+\kappa_2)^2}\Big)\,.
   \end{equation}
Here $F_1$ is the Appell function defined in \eqref{AF1},
   \[\gamma_D:=(4\pi)^{-D/2}\,\Gamma\lr{2-D/2},\qquad
             \kappa_1^2 = (\bm y-\bm q/2)^4, \quad {\rm and} \quad \kappa_2^2 = (\bm y+\bm q/2)^4. \]
Hence, for combinations appearing in the arguments of $F_1$ in \eqref{Iff1} we have
   \begin{equation} \label{AKA}
      \kappa_2-\kappa_1 = 2(\bm y\cdot\bm q) \qquad {\rm and} \qquad \kappa_1+\kappa_2=2(y^2+q^2/4)
   \end{equation}
while $\kappa_1\kappa_2=(y^2+q^2/4)^2-(\bm y\cdot\bm q)^2$.

For generic $\kappa_1$, $\kappa_2$ and integer $D$, \eqref{Iff1} gives the following integrands for the remaining integration over
$\bm y$ in \eqref{Ipq} \cite[Sec. V]{Sh07}, cf. \cite[Sec. IV]{DavDel98}. Apart from the trivial marginal cases
$J_0(p;\kappa_1,\kappa_2)=(\kappa_1\kappa_2)^{-2}$ and $J_{D\nearrow 4}(p;\kappa_1,\kappa_2)=1/[8\pi^2(4-D)]+ \mathscr O(1)$,
these are
   \begin{align}\label{Id1}
      & J_1(p;\kappa_1,\kappa_2) = \frac{\kappa_1+\kappa_2}{2\kappa_1\kappa_2}\,\frac{1}{p^2+(\kappa_1+\kappa_2)^2}, \\\label{Id2}
      & J_2(p;\kappa_1,\kappa_2) = \frac1{2\pi \, \sqrt\Delta}\ln\frac{p^2+\kappa_1^2
                                 + \kappa_2^2+\sqrt{\Delta}}{2\kappa_1\kappa_2}\,,\\ \nonumber
      & J_3(p;\kappa_1,\kappa_2) = \frac1{4\pi p} \arctan\frac{p}{\kappa_1+\kappa_2}.
   \end{align}
In \eqref{Id2},
   \begin{equation} \label{Delta}
      \Delta=\left((\kappa_1+\kappa_2)^2+p^2\right)\left((\kappa_2-\kappa_1)^2+p^2\right)
   \end{equation}
is the same as in \cite[Eq.(2.16)]{Tar08}; this is the counterpart of the triangle K\"allen function of \cite[Eq.(21)]{BDS96}.

In the following, along with \eqref{Id1}, we shall use the special case of \eqref{Iff1} with $\kappa_1=0$ and $\kappa_2=\kappa$,
   \begin{equation} \label{IF0}
      J_D(p;0,\kappa) = -\gamma_D\kappa^{D-4}\, _2F_1\Big(2{-}\frac{D}{2},1;\frac{D}{2};-\frac{p^2}{\kappa^2}\Big).
   \end{equation}
The $_2F_1$ function here is ubiquitous in the literature on Feynman integrals since, presumably, \cite[Eq.(16)]{BG72}
and appears practically in all references to be cited in the next section.

\subsection{Formula \eqref{Iff1}, related references, and a quadratic transformation of $F_1$}

The explicit formula \eqref{Iff1} does not follow from the result \cite[Eq.(20)]{BD91} for the more general integral
$J_D(p;\alpha,\beta;m_1,m_2)$ where $\alpha$, $\beta\ge1$ denote the (higher) powers of denominators appearing in \eqref{IP01}.
In the special case $\alpha=\beta=1$ considered here, the linear combination of the Appell functions $F_4$ in \cite[Eq.(20)]{BD91}
directly reduces to a linear combination of the Gauss hypergeometric functions\footnote{Similar expressions have been
obtained in \cite{Kre92} by avoiding standard techniques, in particular the Feynman parametrization.} $_2F_1$. The same happens also to
\cite[Eq.(21)]{BD91}. These reductions have been noticed readily in \cite{BDS96}.

The elegant approach of functional equations by Tarasov also leads \cite[Eqs.(2.18)-(2.22)]{Tar08} to an analytic expression for
$J_D(p;m_1,m_2)$ in the form of a weighted sum of $J_D(s_{13};m_1,0)$ and $J_D(s_{23};0,m_2)$ in agreement with \cite{BD91,BDS96}
(see \eqref{IF0} for $J_D$ when one of its masses vanishes). The arguments $s_{13}$ and $s_{23}$ are complicated functions
of $m_1, m_2$ involving $\Delta$ from \eqref{Delta} along with its square root.

On the other hand, in \cite[Eq.(1.6)]{KSS82} and, quite recently, in \cite[Eq.(7)]{KT12} the results for the integral
\eqref{IP01} and  $J_D(p;\alpha,\beta;m_1,m_2)$, respectively, have been obtained in terms of a single Appell function $F_1$, similarly
as in \eqref{Iff1}. However, the arguments of that $F_1$ functions are directly related to roots $r_{\mp}$ of quadratic polynomials
which naturally arise in calculations employing the Feynman parametrization (see e.g \cite[Eqs.(7)-(9)]{KT12}). Thus, the external
integration variable $\bm y$, via masses $\kappa_1$ and $\kappa_2$ of \eqref{IP01}, enters these roots $r_{\mp}$ in a quite
complicated manner. There is in fact no chance to do the subsequent $\bm y$ integration by using Appell functions $F_1$ with
such arguments in the integrand as required by \eqref{Ipq}.

As an illustration, we quote the result for \eqref{IP01} from \cite[Eq.(1.6)]{KSS82}. In our notation
   \begin{equation} \label{KSS}
      J_D(p;\kappa_1,\kappa_2) = \gamma_D \, \kappa_1^{D-4}\,F_1\Big(1;2-\frac d2,2-\frac d2;2;\frac1r_-,\frac1r_+\Big)\,,
   \end{equation}
where $r_{\mp}=(p^2+\kappa_2^2-\kappa_1^2\mp\sqrt\Delta)/2p^2$, and again $\Delta=(p^2+\kappa_2^2-\kappa_1^2)^2+4\kappa_1^2p^2$.

By contrast, the arguments of the $F_1$ function in \eqref{Iff1} have simple and transparent dependencies on $\bm y$ as can be easily
seen from \eqref{AKA}. This has been achieved in \cite{Sh07} by taking special care of symmetry of the first integral in \eqref{IP01}
and avoiding such standard technical tools as Feynman or Schwinger parametrization. The structure of arguments of the Appell function
$F_1$ in \eqref{Iff1} allows its term-by-term integration by using its series definition \eqref{AF1} without producing any
spurious divergence. The absolute value $p$ of the external momentum also enters \eqref{Iff1} in a non-complicated
manner. This suggests an eventual possibility to use this result in higher-loop calculations of standard Feynman integrals.

Apparently, a generalization of the result \eqref{Iff1} to $J_D(p;\alpha,\beta;\kappa_1,\kappa_2)$ with higher integer powers of
denominators $\alpha$ and $\beta$ should be possible by differentiating \eqref{IP01} and \eqref{Iff1} with respect to $\kappa_1^2$ and
$\kappa_2^2$ and using the differentiation formulae for the Appell function $F_1$ derived quite recently in \cite{BS12}.

At $m=1$, a term-by-term integration of \eqref{Iff1} over $y$ has already been performed in a study of scaling functions
at the Lifshitz point in \cite[Sec.(5.3)]{RDS11}. The resulting series expansions for $I_{D,1}(p,q)$ (see \cite[Eqs.(5.69)-(5.73)]{RDS11})
involve the Clausenian hypergeometric functions $_3F_2$ with unit argument and negative integer parameter differences between its
numerator and denominator parameters studied very recently by one of the present authors and Srivastava \cite{SS15}.
The series expansion with the same structure can be obtained also in the case of $I_{1,m}(p,q)$ considered throughout the present
paper, whereas at $D=3$ the analytic expression is \cite{WIP}
   \[ I_{3,m}(p,q) = 8(16\pi)^{-\E}\Gamma(2-\E)(4p^2+q^4)^{-1+\E/2}\,
	                  {}_2F_1\lrB{1-\frac{\E}{2},\frac{\E}{2}; \frac32;\frac{4p^2}{4p^2+q^4}}\,, \]
where $\E=m/2+1\in[1,2[$.

Finally, by comparing the expressions \eqref{Iff1} and \eqref{KSS} for the same integral \eqref{IP01} we come to
a non-trivial relation between involved Appell functions $F_1$. This is the special case of the quadratic transformation
\cite[Eq.(4.2)]{Carl76} expressed in our notation as
   \[ F_1\Big(a;b,b;2a;\frac{1}{r_-},\frac{1}{r_+}\Big)\!\! = \!
         \Big(\frac{x+y}{2x}\Big)^{-2b}\!F_1\left(b;a,b-a+\frac 12;a+\frac12;
         \frac{-1}{(x+y)^2},\Big(\frac{x-y}{x+y}\Big)^2\right)\]
where $r_{\mp}=(1+y^2-x^2\pm\sqrt\Delta)/2\,$ and $\,\Delta=(1+y^2-x^2)^2+4x^2$, taken at $a=1$ and $b=2-D/2$.

Similar transformations are discussed in contemporary mathematical literature: see \cite{SR14} and its references. We suppose, they
should be useful in practical Feynman-integral calculations.

\section{The case $D=1$: splitting $J_1(p;\kappa_1,\kappa_2)$} \label{SAS}

The simplest way to proceed at $D=1$ is to split the result \eqref{Id1} of the integration over $\bm x$ via
   \[ J_1(p;\kappa_1,\kappa_2) = \frac12\Big(\frac{1}{\kappa_1}+\frac{1}{\kappa_2}\Big)\,
                                 \frac{1}{p^2+(\kappa_1+\kappa_2)^2}\,.\]
Hence we rewrite the original integral $I(p,q)$ as $I_{1,m}(p,q)=\tfrac12 \lr{I^-+I^+}$ where
   \begin{equation} \label{S3}
      I^\mp\equiv\int_{\bm y}^{(m)}\frac{1}{(\bm y\mp\bm q/2)^2}\,\frac{1}{p^2+4(y^2+q^2/4)^2}\,.
   \end{equation}
The integrals $I^\mp$ are independent of the signs at $\bm q$ and of its direction in $\mathbb R^m$.
Similarly as in \eqref{IP01} itself, shifting the variable $\bm y\mp\bm q /2\to\bm y$, we obtain
   \begin{equation} \label{S4}
      I_{1,m}(p,q) = \int_{\bm y}^{(m)}\frac{1}{y^2}\,\frac{1}{p^2+4a^2}\qquad\mbox{with}\qquad
      a:=(\bm y\pm\bm q/2)^2+q^2/4\,.
   \end{equation}
The apparent singularity of the integrand is now shifted to the origin. However, it is integrable because the dimension of the
integral over $\bm y$ is high enough: at $D=1$ we consider $m \in ]2,6[$.

The simplest possibility to proceed is to factorize the second denominator in \eqref{S4} via $p^2+4a^2 =
(2a-{\rm i}p)(2a+{\rm i}p)$ and to apply the partial fraction expansion
   \begin{equation} \label{S5}
      \frac{1}{p^2+4a^2}=\frac1{4p\,{\rm i}}\lrs{\frac{1}{a-{\rm i}p/2}-\frac1{a+{\rm i}p/2}}\,.
   \end{equation}
This splits the the integrand in \eqref{S4} into a difference of two terms and we have
   \begin{equation} \label{S7}
      I_{1,m}(p,q) = \frac1{4p\,{\rm i}}\lr{j_- - j_+} = \frac1{4p\,{\rm i}}\lr{j_--j_-^*}=\frac1{2p}\;\Im\,j_-
   \end{equation}
with
   \[ j_\mp\equiv \int_{\bm y}^{(m)}\frac{1}{y^2}\,\frac{1}{(\bm y\pm\bm q /2)^2+q^2/4\mp {\rm i}p/2}\,. \]
Now, $j_\mp$ can be identified with the second integral of \eqref{IP01} with $\bm x\mapsto \bm y$, $D\mapsto m$, $\kappa_1^2\mapsto 0$,
$\kappa_2^2\mapsto\kappa^2\mapsto q^2/4-ip/2$, and $\bm p\mapsto\pm\bm q/2$. Hence it is given by \eqref{IF0} with indicated changes
of its parameters:
   \[ j_\mp = J_m(q/2;0,\kappa) = -\frac{16}{(16\pi)^{\frac m2}}\,\Gamma\lrB{1-\frac m2}(q^2-{\rm i}2p)^{\frac m2-2}\,
                                  {}_2F_1\lrB{2{-}\frac{m}{2},1;\frac{m}{2};\frac{-q^2}{q^2\!-\!{\rm i}2p}}. \]
Inserting this result with $m=2+2\E$ into \eqref{S7} we finish the proof of the

\begin{proposition} For $D=1$ and real $m \in ]2;6[$, the integral \eqref{Ipq} is
   \begin{equation} \label{IFN}
      I_{1,m}(p,q) = -\frac1p\,\frac{\Gamma(-\epsilon)}{2^{1+4\epsilon}\,\pi^{1+\epsilon}}\;
			               \Im\,(q^2-{\rm i}2p)^{-1+\E}\,_2F_1\lrB{1,1-\E;1+\E;\frac{-q^2}{q^2\!-\!{\rm i}2p}}
   \end{equation}
where $\E=m/2-1\in ]0;2[$.
\end{proposition}

This is one of the main results of the present article. Expressed in terms of $_2F_1$, it provides an analytical continuation for
the function $I_{1,m}(p,q)$ outside of the convergence region of the integral \eqref{Ipq} at $D=1$.

At $q=0$, \eqref{IFN} reduces to
   \[ \hat I_{1,m}(p,0) = C_1 \frac{(2p)^{-2+\epsilon}}{1-\epsilon}\;\Im \;(-i)^{-1+\E}
                        = C_1 \frac{(2p)^{-2+\epsilon}}{1-\epsilon}\,\cos\frac{\pi\epsilon}{2}\,, \]
in agreement with \eqref{S00} and \eqref{I10}.

As $p \to 0$, the imaginary part of the complex function in \eqref{IFN} vanishes. To obtain the correct result for
$I_{1,m}(0,q)$ one has to expand it to first order in $p$. This cancels the overall factor $p^{-1}$ in \eqref{IFN}
yielding $I_{1,m}(0,q)=q^{-4+2\E}\,I_{1,m}(0,1)$ with $I_{1,m}(0,1)$ given in \eqref{I10}.

\subsection{Special cases}

Let us consider the implications of the formula \eqref{IFN} at integer values of $m \in [2;6]$. To simplify the exposition we consider
the reduced function
   \begin{equation} \label{IFR}
      \hat I_{1,m}(p,q)= C_1^{-1}I_{1,m}(p,q) = \Im \, \frac{(q^2-{\rm i}2p)^{\E-1}}{2(1-\epsilon)p}\, {}_2F_1\lrB{1,1-\E;1+\E;
			                   \frac{-q^2}{q^2-{\rm i}2p}}\,,
   \end{equation}
and introduce the short-hands $z := q^2-{\rm i}2p$ and $y := -q^2/z$.

\begin{itemize}
   \item{At $m=2$ we have $\epsilon = 0$, the equation \eqref{IFR} reduces to
         \[ \hat I_{1,2}(p,q) = \frac1{2p}\;\Im\,z^{-1}\,_2F_1\lr{1,1;1;y}=\frac1{2p}\;\Im\;\frac{z^{-1}}{1-y}
                        = \frac14\,\frac1{p^2+q^4}\,, \]
      and hence
         \[ I_{1,2+0}(p,q) = \frac1{4\pi\epsilon}\,\frac1{p^2+q^4}+ \mathscr O(1)\,.\] }
   \item{At $m=3$, $\epsilon = 1/2$ and we have
         \[ \hat I_{1,3}(p,q) = p^{-1}\;\Im\, z^{-1/2}\,_2F_1\lrB{1,\frac12;\frac32;y}
                        = -\frac1{4pq}\ln\left|\frac{\sqrt z+{\rm i}q}{\sqrt z-{\rm i} q}\right|^2\,. \]
      Some straightforward algebra leads to
         \begin{equation} \label{M3}
            \hat I_{1,3}(p,q) = -\frac1{4pq}\, \ln \frac{q^2\sqrt{4p^2+q^4}+p^2-pq\sqrt2\,\sqrt{\sqrt{4p^2+q^4}+q^2}}{p^2+q^4}\,.
         \end{equation}  }
   \item{At $m=4$ (or $\epsilon = 1$) the result quoted in \eqref{I41} has been obtained by different means and written down,
      without derivation, in \cite[Eq.(68)]{SPD05}. Now we obtain it from \eqref{IFR}. Here, as in the case of $p=0$,
      the imaginary part of the function on the right vanishes at $\epsilon=1$ and we need to calculate
         \[ \hat I_{1,4}(p,q) = \frac1{2p}\, \lim_{\alpha\to0}\alpha^{-1}\,\Im\, z^{-\alpha}\,_2F_1\lr{\alpha,1;2-\alpha;y}
				                        \qquad\mbox{with}\qquad \alpha:=1-\epsilon. \]
      Taking into account that $(\alpha)_n = \alpha \cdot (n-1)!+ \mathscr O(\alpha^2)$ and proceeding in the standard way we
			obtain\footnote{It is worth mentioning that the Mathematica package HypExp \cite{HM06} allows to expand arbitrary
			$_rF_{r-1}$ hypergeometric functions to any order in a small quantity around integer parameters.}
         \begin{equation} \label{HI4}
            \hspace{-1mm}\hat I_{1,4}(p,q) = \frac1{2p}\Im \lrS{\frac{1-y}y\ln(1-y)-\ln z}
			                               = \frac1{2p}\lrS{\frac{p}{q^2}\,\ln|1-y|^2-(2\varphi_1+\varphi)}
         \end{equation}
      where $\varphi_1 = {\rm Arg}(1-y) = \arctan\big[pq^2/(2p^2+q^4)\big]$ and $\varphi = {\rm Arg}\;z=-\arctan\big[2p/q^2\big]$
      are the principal arguments of $1-y$ and $z$. Applying repeatedly the formula (see e.g. \cite[Sec. I.3.5]{PBM1})
         \[ \arctan x-\arctan y=\arctan\frac{x-y}{1+xy}\,,\quad\quad xy>-1, \]
      to the sum of arguments $2\varphi_1+\varphi$ in \eqref{HI4} we obtain
         \begin{equation} \label{M4}
            \hat I_{1,4}(p,q) = \frac1{2p}\,\bigg[\frac{p}{q^2}\ln{p^2+q^4\over p^2+q^4/4}
                              + \arctan\frac{2p^3}{q^2(3p^2+q^4)}\bigg]\,.
         \end{equation}
      With $C_1(\epsilon=1)=1/(16\pi^2)$ and $p=1$, this matches the result \eqref{I41}.\\ }
   \item{At $m=5$, $\E=3/2$ and the equation \eqref{IFR} simplifies to
         \begin{align*}
	          \hat I_{1,5}(p,q) &= -p^{-1}\;\Im\, z^{1/2}\,_2F_1\lrB{1,-\frac12;\frac52;y} \\
	                            &= -p^{-1}\;\Im\, z^{1/2}\,\frac{3}{8y}\lrS{1+y-\frac{(1{-}y)^2}{2\sqrt y}\,
														    \ln\frac{1+\sqrt y}{1-\sqrt y}}\,.
	       \end{align*}
      This can be written, for example, as
         \begin{equation} \label{M5}
            \hat I_{1,5}(p,q) = -\frac3{4q} \, \Re\,\lrS{ \sqrt\zeta- \frac1u\, (1-u{\rm i})^2\,\ln\frac{\sqrt\zeta+{\rm i}}
                                {\sqrt\zeta-{\rm i}}}\,,
         \end{equation}
      which is the scaling form implied by the second expression in \eqref{SCI} with $p/q^2\mapsto u$ and $\zeta=1-2u{\rm i}$.
      Now, \eqref{M5} can be re-expressed in a form similar to \eqref{M3}, \eqref{M4}. \\}
   \item{Finally, at the marginal value $m=6$ (or $\E=2$) we obtain
         \[ \hat I_{1,6}(p,q) = -\frac{1}{2p}\;\Im\, z\,_2F_1\lrB{-1,1;3;y} = -\frac{1}{2p}\;\Im\, z\lrB{1-\frac{y}{3}} = 1\,. \]
      Hence
         \[ I_{1,6-0}(p,q) = \frac{\Gamma(2-\epsilon)}{512\pi^3}
                           = \frac1{512\pi^3\varepsilon}+ \mathscr O(1) \quad {\rm with} \quad \varepsilon=2-\E. \]
      Here, the $1/\varepsilon$ contribution matches the pole terms of both $I_{1,6-2\varepsilon}(1,0)$ and $I_{1,6-2\varepsilon}(0,1)$
      from \eqref{I10} when $\varepsilon \to 0$.}
 \end{itemize}

\begin{remark}
By splitting the denominator of the integrand in \eqref{S4} we also expressed the sum $p^2+q^4/4$ originally present in \eqref{S3} as
the product $(q^2/2-{\rm i}p)(q^2/2+{\rm i}p)$. The resulting factors landed in the arguments of functions in the final result of
this section, \eqref{IFN}. Most important here is the complex-argument Gauss hypergeometric function ${}_2F_1(z)$. In the
previous paragraphs we sketched its complex expansions in several simplifying cases where the parameter $m$ is integer. As expected,
the original combination $p^2+q^4/4$ is reconstructed in \eqref{M3}, \eqref{M4}, \eqref{M5}. The same expression is also present in  evaluations related to the special case $m=1$ with $D=2$ and $D=3$ quoted in the Introduction. \hfill $\Box$
\end{remark}

It is natural to ask here whether we can do a similar calculation in generic case given by \eqref{IFN}. That is, can we do the
complex expansion of a Gauss function $_2F_1$ by finding explicit expressions for its real and imaginary parts? We address this
general question it in the following section.

\section{Complex expansion of $_2F_1(a,b;c;z)$}\label{SB}

Our aim in this section is to derive explicit expressions for real and imaginary parts of the Gauss hypergeometric function with real parameters and complex argument. We shall do it in different ways by producing several series
representations and a Laplace-type integral representation. We hope that our results are of independent interest and
can be considered as a contribution to the theory of special functions.

Consider $_2F_1(a,b;c;z)$ with real parameters $a$, $b$, and $c$ and the complex argument $z$. We define its real and imaginary
parts via
   \[ X + {\rm i}Y:=\Re\, {}_2F_1(a,b;c;z) + {\rm i}\,\Im \,_2F_1(a,b;c;z) \]
where $z = x + {\rm i}\,y=|z|\,\e^{{\rm i}\varphi}\in \mathbb C$, finite $x, y\in \mathbb R\setminus\{0\}$,
$|z|=\sqrt{x^2+y^2}$, and $\varphi=\arctan(y/x)$. Let us start with the series representations for $X$ and $Y$.

\begin{proposition}\label{P2}
For all $a, b \in \mathbb R,\, c\in \mathbb R\setminus \mathbb Z_0^-$ and for all $|z|<1$, or for $c-a-b>0$ when $|z|=1$, we have
   \begin{align} \label{XXF}
      X &= \sum_{k \geq 0}\frac{(a)_k(b)_k}{(c)_k\, k!} \Big(\frac{|z|^2}x\Big)^k
           \,{}_2F_1\Big(\frac{k}2,\frac{k+1}2;\frac12;-\frac{y^2}{x^2}\Big), \\
      Y &= y\, \frac{|z|^2}{x^2}\,\frac{ab}c \sum_{k\ge0}\frac{(a+1)_k(b+1)_k}{(c+1)_k \,k!}\,\Big(\frac{|z|^2}x\Big)^k
           \,{}_2F_1\Big(\frac{k+2}2,\frac{k+3}2;\frac32;-\frac{y^2}{x^2}\Big). \label{YYF}
	 \end{align}
\end{proposition}

\begin{proof} Let $|y| \leq |x|$ (otherwise we consider $w = {\rm i}z$).
Using the polar form of $z$, from the series definition \eqref{A4}
for $_2F_1$ we read off
   \begin{equation} \label{XX}
      X\!=\!\sum_{k\ge 0}\frac{(a)_k(b)_k|z|^k}{(c)_k\, k!} \cos\!\Big(k\arctan\frac{y}{x}\Big), \,\,
		  Y\!=\!\sum_{k\ge 0}\frac{(a)_k(b)_k|z|^k}{(c)_k\, k!} \sin\!\Big(k\arctan\frac{y}{x}\Big).
   \end{equation}
Using the entry \cite[7.3.3.1]{PBM3} with $a\mapsto k/2$
transforms immediately \eqref{XX} to \eqref{XXF}.
By \cite[7.3.3.2]{PBM3} with $a\mapsto (k+1)/2$ we obtain
   \[ Y = \frac yx\,\sum_{k \geq 1}\frac{(a)_k(b)_k}{(c)_k\, (k-1)!} \Big(\frac{|z|^2}x\Big)^k
          \,{}_2F_1\Big(\frac{k+1}2,\frac{k}2+1; \frac32;-\frac{y^2}{x^2}\Big). \]
Shifting here the summation index by $1$ leads to the asserted formula \eqref{YYF}.
\end{proof}

In \eqref{XXF} and \eqref{YYF} we obtained $X$ and $Y$ in the form of series expansions in
powers of $|z|^2/x$, while the ratios $-y^2/x^2$ appear as arguments of $_2F_1$ functions in expansion coefficients.
In the next proposition we exchange the roles of $|z|^2/x$ and $-y^2/x^2$.
This time the Clausenian function ${}_3F_2(|z|^2/x)$ will give the expansion coefficients at powers of $-y^2/x^2$.

\begin{proposition} For all $a, b \in \mathbb R,\, c\in \mathbb R\setminus \mathbb Z_0^-$ and for all $|z|<1$, or for $c-a-b>0$ when
$|z|=1$, we have
   \begin{align} \label{W3}
      X &= 1+ \frac{|z|^2}x\,\frac{ab}c \sum_{k \geq 0}
             \,{}_3F_2\left( \begin{array}{c} a+1,\,b+1,\,2k+1\\
             c+1,\,2\end{array} ;\frac{|z|^2}x\right)
             \Big(-\frac{y^2}{x^2}\Big)^k\,,\\ \label{W4}
      Y &= y\, \frac{|z|^2}{x^2}\,\frac{ab}c \sum_{k \geq 0}
             \,{}_3F_2\left( \begin{array}{c} a+1,\,b+1,\,2k+2\\
             c+1,\,2\end{array} ;\frac{|z|^2}x\right)
             \Big(-\frac{y^2}{x^2}\Big)^k\, .
   \end{align}
\end{proposition}

\begin{proof}
Under the same conditions as in Proposition \ref{P2}, let us write $z=|z|^2/z^*$ where the asterisk denotes complex conjugate.
Hence
   \[ z^n = \frac{|z|^{2n}}{x^n}\, \Big( 1-{\rm i}\frac yx\Big)^{-n}
					= \frac{|z|^{2n}}{x^n}\, {}_1F_0 \Big(n; -; {\rm i}\frac yx\Big)\,. \]
Then, by suitable series transformations including an exchange of the summation order,
   \begin{align}\nonumber
      {}_2F_1&(a,b;c;z) = 1 + \sum_{n \geq 1} \frac{(a)_n(b)_n}{(c)_n\, n!}\,\frac{|z|^{2n}}{x^n}\,
                          \sum_{k \geq 0} \frac{\Gamma(n+k)}{\Gamma(n)\,k!}\,
                          {\rm i}^k\, \Big( \frac yx\Big)^k
      \\\nonumber
                       &= {}_2F_1\left(|z|^2/x\right) + \frac{|z|^2}x\,\frac{ab}c \sum_{k \geq 1} {\rm i}^k\,\Big( \frac yx\Big)^k
                          \sum_{n \geq 0} \frac{(a+1)_n(b+1)_n(k+1)_n}{(c+1)_n(2)_n\,n!}\, \frac{|z|^{2n}}{x^n}
      \\\label{ER}
                       &= {}_2F_1\left(|z|^2/x\right) + \frac{|z|^2}x\,\frac{ab}c \,\sum_{k \geq 1} {\rm i}^k\,\Big(\frac yx\Big)^k\,
                          \,{}_3F_2\left( \begin{array}{c} a+1,\,b+1,\,k+1\\ c+1,\,2 \end{array} ;\frac{|z|^2}x\right)\,.
   \end{align}
Taking into account that the last $_3F_2$ function simplifies at $k=0$, we end up with
   \begin{equation} \label{EUW}
      {}_2F_1(a,b;c;z)=1+\frac{|z|^2}x\,\frac{ab}c\, \sum_{k \geq 0} {\rm i}^k\,\Big( \frac yx\Big)^k\,
			                      \,{}_3F_2\left( \begin{array}{c} a+1,\,b+1,\,k+1\\ c+1,\,2\end{array} ;\frac{|z|^2}x\right)\,.
   \end{equation}
Separating sums over the odd and even $k$ we come directly to \eqref{W3} and \eqref{W4}.
\end{proof}

The next proposition gives Laplace-type integral representations for $X$ and $Y$.

\begin{proposition}
Under the same conditions as in {\rm Proposition \ref{P2}}, we have
   \begin{align*}
      X &= 1 + |z|^2\,\frac{ab}{c} \int_0^\infty {\rm e}^{-xt}\,
           \cos(yt)\,{}_2F_2\left( \begin{array}{c} a+1,\,b+1\\c+1,\,2\end{array}
           ; |z|^2\,t\right)\, {\rm d}t\,,\\
      Y &= |z|^2\,\frac{ab}{c} \int_0^\infty {\rm e}^{-xt}\,
           \sin(yt)\,{}_2F_2\left( \begin{array}{c} a+1,\,b+1\\c+1,\,2\end{array}
           ; |z|^2\,t\right)\, {\rm d}t\, .
   \end{align*}
\end{proposition}

\begin{proof}
Taking into account the integral representation \cite[p.115, Exersize 11]{AAR}
   \[ {}_3F_2\left( \begin{array}{c} a,\,b,\,c\\ f,\,g\end{array} ; w \right)
       = \frac1{\Gamma(c)}\int_0^\infty {\rm e}^{-t}t^{c-1}\,
         {}_2F_2\left( \begin{array}{c} a,\,b\\f,\, g\end{array} ;wt\right)\, {\rm d}t
\qquad (\R(c)>0)\]
with $c=k+1$, we deduce from the series \eqref{EUW} the both statements.
\end{proof}

\section{Gauss $_2F_1$ {\it {\bf versus}} Horn $H_4$}\label{SB1}

In this section we shall consider a special situation in which one of the numerator parameters of $_2F_1(a,b;c;z)$, say $b$, equals to $1$.
We shall show that in this case $X$ and $Y$ can be expressed in terms of the Horn function of two variables $H_4$
defined by the double series in \eqref{H4D}.

Before proceeding we recall that the Horn function $H_4$ can be expressed as a single power series involving
$_2F_1$ \cite[Eq. (3.9)]{Deshp74}:
   \begin{equation} \label{D39}
      H_4(\alpha,\beta;\gamma,\delta;\,s,t) = \sum_{n \geq 0}\frac{(\alpha)_n(\beta)_n}{(\delta)_n}\,
          \,_2F_1\Big(\frac{n+\alpha}2,\frac{n+\alpha+1}2;\gamma;4s\Big)\,\frac{t^n}{n!}\,.
   \end{equation}
Indeed, using the Legendre duplication formula $\Gamma(z)\,\Gamma(z+\tfrac12)=\sqrt{\pi}\,2^{1-2z}\,\Gamma(2z)$ $(\Re(z)>0)$ in the
right--hand side we reproduce the double--series definition \eqref{H4D}. The parameter space for \eqref{D39} is
$\alpha, \beta \in \mathbb C; \gamma, \delta \in \mathbb C \setminus \mathbb Z^0_-$, while the convergence domain, according to
Horn's theorem \cite[p. 56 {\em et seq.}]{SriKar}, is given by $\{ (s,t) \colon 2\sqrt{|s|} + |t|<1\}$. We note also an alternative
form of writing $H_4$ in terms of $_2F_1$,
   \begin{equation} \label{DX}
      H_4(\alpha,\beta;\gamma,\delta;\,s,t) = \sum_{n \geq 0}\frac{(\alpha)_{2n}}{(\gamma)_n}\,
          \,_2F_1\left(\alpha+2n,\beta;\delta;t\right)\frac{s^n}{n!}\,.
   \end{equation}
Let us also recall \cite[Eq. (3.2)]{Deshp74} that the Horn function $H_4$ is expressible in terms of the the Appell function
$F_2$ (see \eqref{AF2}) {\it via}
   \begin{equation}\label{HVF}
      H_4(\alpha, \beta; \gamma, \delta; s,t) = (1+2\sqrt s)^{-\alpha}F_2 \Big(\alpha,\beta,\gamma-\frac12;\delta,2\gamma-1;
          \frac{t}{1+2\sqrt s},\frac{4\sqrt s}{1+2\sqrt s}\Big).
   \end{equation}
Moreover, when the parameters of the Horn function $H_4$ are related by $\beta=\alpha-\gamma+1$, we find its expression in terms
of the Appell function $F_4$ (see \eqref{AF4}):
\footnote{A more special relation \cite[Eq.(3.4)]{Deshp74} is misprinted: The both parameters
$\gamma$ and $\sigma+b'$ of $H_4$ should be replaced there by $(\alpha+1)/2$.}
   \begin{align*}
      &H_4(\alpha, \beta; \alpha-\beta+1, \delta; s, t)\\
      &= \Big(\frac{1+\sqrt{1-4 s}}{2}\Big)^{-\alpha} F_4\Big(\alpha,\beta;\delta,\alpha-\beta+1;\frac{2t}{1+\sqrt{1-4 s}},
         \frac{1-\sqrt{1-4 s}}{1+\sqrt{1-4 s}}\Big).
   \end{align*}
This can be proved by employing the quadratic transformation \cite[7.3.1.98]{PBM3} to the Gauss function in (\ref{D39}).
{Note that the relation of the parameters $\gamma=\alpha-\beta+1$ in $F_4$ is typical for another quadratic transformation of
Gauss functions, \cite[7.3.1.54]{PBM3}.}

\subsection{Real and imaginary parts of $_2F_1(a,1;c;z)$}

Let us return to our discussion of the complex expansion of the  Gauss function $_2F_1(a,b;c;z)$.
Comparing its real and imaginary parts \eqref{XXF} and \eqref{YYF} with the series
{definition of the Horn function $H_4$} \eqref{D39} we conclude that
   \begin{align} \label{X1H}
      \Re\, {}_2F_1(a,1;c;z) &= 1+\frac ac\,\frac{|z|^2}x\, H_4\Big(1,a+1;\frac12,c+1;-\frac{y^2}{4x^2},\frac{|z|^2}x\Big), \\
                                \label{Y1H}
      \Im\, {}_2F_1(a,1;c;z) &= \frac ac\,\frac{y|z|^2}{x^2}\,H_4\Big(2,a+1;\frac32,c+1;-\frac{y^2}{4x^2},\frac{|z|^2}x\Big).
   \end{align}
This implies that the result \eqref{IFN} for the integral $I_{1,m}(p,q)$
is expressible as a linear combination of
Horn functions $H_4$ from \eqref{X1H} and \eqref{Y1H} with $a{=}1{-}\E$, $c{=}1{+}\E$, and
   \[ z = -\frac{q^2}{q^2-{\rm i}\,2p} = - \frac{q^4}{q^4+4p^2} - {\rm i}\,\frac{2pq^2}{q^4+4p^2} := x+{\rm i}\,y,
	        \qquad |z|^2 = \frac{q^4}{q^4+4p^2}\, . \]
Namely, we have to deal with
   \begin{align*}
	    \Re\, {}_2F_1\Big(1-\E,1;1+\E;\frac{-q^2}{q^2-{\rm i}\,2p}\Big)
			      &= 1-\frac{1-\E}{1+\E}\, H_4\Big(1, 2-\E, \frac12, 2+\E; -\frac{p^2}{q^4}, -1\Big) \\
      \Im\, {}_2F_1\Big(1-\E,1;1+\E;\frac{-q^2}{q^2-{\rm i}\,2p}\Big)
			      &= -2\,\frac{1-\E}{1+\E}\,\frac{p}{q^2}\, H_4\Big(2, 2-\E, \frac32, 2+\E; -\frac{p^2}{q^4}, -1\Big)\,.
  \end{align*}
{In the next section we shall remove the constraint $b=1$ in considering further relations
between the functions $_2F_1$ and $H_4$.

\subsection{$_2F_1$ with arbitrary parameters}

Let us reverse the problem discussed before and start now from the Horn functions $H_4$ with parameters $\gamma\in\{\frac12,\frac32\}$.
Such consideration leads us to the}

\begin{proposition}\label{P5}
Under the same assumptions on the parameters as above, $x\in \mathbb R_+$, and $y\in \mathbb R$,
the following relations hold:
   \begin{align}\label{PPM}
	    H_4\lrB{\alpha,\beta;\frac12,\delta;-x,y} &= \Re\, (1\pm{\rm i}2\sqrt x)^{-\alpha}\,
	          {}_2F_1\lrB{\alpha,\beta;\delta;\frac{y}{1\pm{\rm i}2\sqrt x}}, \\\nonumber
      H_4\lrB{\alpha,\beta;\frac32,\delta;-x,y} &=
      \frac{\mp 1}{2(\alpha{-}1)\sqrt x}\Im\,(1\pm{\rm i}2\sqrt x)^{1-\alpha}
	          {}_2F_1\lrB{\alpha-1,\beta;\delta; \frac{y}{1\pm{\rm i}2\sqrt x}}.
	 \end{align}
\end{proposition}

\begin{proof}
Consider \eqref{D39} with $\gamma=1/2$ and $\gamma=3/2$, \emph{negative} $s\mapsto-x$, and $t\mapsto y$. Thus, the Gaussian function
on the right is given by $_2F_1(\mu, \mu+\frac12; \frac12;-4x)$ or $_2F_1(\mu, \mu+\frac12; \frac32;-4x)$
with $\mu\mapsto(\alpha+n)/2$. For these two cases the known formulae \cite[7.3.1.106 and 7.3.1.107]{PBM3} read
   \begin{align}\label{IXI}
      {}_2F_1\Big(\mu, \mu+\frac12; \frac12; -4x\Big) &= \frac12 \left[ (1-{\rm i}2\sqrt{x})^{-2\mu}
			          + (1+ {\rm i}2\sqrt{x})^{-2\mu} \right],\\ \label{IXY}
      {}_2F_1\Big(\mu, \mu+\frac12; \frac32; -4x\Big) &= \frac{\left[ (1- {\rm i}2\sqrt{x})^{1-2\mu}
			          - (1+ {\rm i}2\sqrt{x})^{1-2\mu} \right]}{{\rm i}4\sqrt{x}(2\mu-1)}\, .
	 \end{align}
We see that the square brackets in the last two equations contain the sum and the difference of certain complex conjugate values
thus giving for these values their real and imaginary parts. Keeping this in mind, we substitute the right-hand sides of \eqref{IXI}
and \eqref{IXY} into \eqref{D39}. Remembering that $\mu=(\alpha+n)/2$ we sum up the resulting series over $n$ and obtain the
asserted formulae. While the summation leading to \eqref{PPM} is straightforward, in deriving the following result we had to take
into account the relation $(\alpha)_n/(\alpha+n-1)=(\alpha-1)_n/(\alpha-1)$.
\end{proof}

The results of the Proposition 5.1 are more general than that of \eqref{X1H} and \eqref{Y1H}. {Here the previous constraint
$b=1$ is removed, and the involved Gauss hypergeometric functions $_2F_1$ are considered with generic real parameters.}

\begin{corollary}
The result \eqref{IFN} for $I_{1,m}(p,q)$ can be expressed in terms of a {\rm single} Horn function $H_4$:
   \[ I_{1,m}(p,q) = \frac{\Gamma(2-\epsilon)}{16^\epsilon\,\pi^{1+\epsilon}\,\epsilon}\;
	                   q^{-4+2\epsilon}\, H_4\lrB{2-\E,1;\frac32,1+\E;\frac{-p^2}{q^4},-1}. \]
The constant in front of the last formula is just $C_1(\E)$ defined in \eqref{I10}.
In view of the relation \eqref{HVF} {between $H_4$ and $F_2$}, this result, as well as \eqref{X1H} and \eqref{Y1H},
can be expressed in terms of of the Appell function $F_2$ (see \eqref{AF2}).
\end{corollary}

Though very compact and elegant, we must admit that such representations for the function $I_{1,m}(p,q)$
{in terms of the Horn function $H_4$ or the Appell function $F_2$} are mainly of academic interest from
the mathematical point of view. Their practical implications seem to be at the moment rather obscure.

{In concluding, we show the following relation of the formula \cite[6.8.1.17]{PBM3} to our calculations. We take this
entry at $p=q=1$ with $t\mapsto{\rm i}t$ and rewrite it as
   \begin{equation}\label{EE}
      (1-{\rm i}t)^{-a}\,_2F_1\Big(a,b;c;\frac{s}{1-{\rm i}t}\Big) =
         \sum_{k \geq 0}\frac{(\alpha)_k}{k!}\,({\rm i}t)^k\,_2F_1\left(a+k,b;c;s\right)
   \end{equation}
where $|t|<1,\;|s|<1$. The series on the right-hand side is similar by its structure to that of \eqref{ER}. Again, its real and
imaginary parts are given by partial sums over even and odd summation indices $k$, respectively.

Let us express the argument of the Gauss function on the right of \eqref{EE} as
   \[ z=x+{\rm i}y \qquad \mbox{with} \quad x:=\frac{s}{1+t^2}\;\, \mbox{and} \quad y:=\frac{ts}{1+t^2}. \]
This implies that $t=y/x$ and $s=|z|^2/x$ with $|z|^2=x^2+y^2$ as usual. Thus, for the real part of \eqref{EE} we obtain
   \[ \Re\, z^a{}_2F_1\left(a,b;c;z\right) = \Big(\frac{|z|^2}{x}\Big)^a \sum_{k \geq 0} \frac{(\alpha)_{2k}}{(\frac12)_k k!}\,
                      \Big(\frac{-y^2}{4x^2}\Big)^k\,_2F_1\Big(a+2k,b;c;\frac{|z|^2}{x}\Big) \]
where we took into account that $(2k)!=4^k(\frac12)_k k!$. The last series matches the right--hand side of the definition \eqref{DX}
of the Horn function $H_4$ in terms of $_2F_1$. Hence we have
   \[ \Re\, z^a{}_2F_1\left(a,b;c;z\right) = \Big(\frac{|z|^2}{x}\Big)^a H_4\Big(a,b;\frac12,c;\frac{-y^2}{4x^2},
	                    \frac{|z|^2}{x}\Big) \]
which is equivalent to (\ref{PPM}). An analogous expression can be obtained for the imaginary part of
$z^a{}_2F_1\left(a,b;c;z\right)$.}

\section*{Acknowledgements}
The authors are grateful to the anonymous referee for valuable comments and suggestions, which led to the substantial improvement
of the article, {and to Professor Yu. A. Brychkov for drawing our attention to the entry \cite[6.8.1.17]{PBM3} in
connection with our work}.


\begin{thebibliography}{99}

\bibitem{SPD05}
Shpot~MA, Pis{'}mak~{\relax Yu}M, Diehl~HW. Large-$n$ expansion for $m$-axial
  {L}ifshitz points. J Phys: Condens Matter. 2005; 17(20): S1947--S1972.

\bibitem{RDS11}
Rutkevich~S, Diehl~HW, Shpot~MA. On conjectured local generalizations of
  anisotropic scale invariance and their implications. Nucl Phys B. 2011; 843(1): 255--301;
	{E}rratum: Nucl Phys B. 2011; 853: 210-211.

\bibitem{SP12}
Shpot~MA, Pis{'}mak~{\relax Yu}M. Lifshitz-point correlation length exponents
  from the large-$n$ expansion. Nucl Phys B. 2012; 862(1): 75--106.

\bibitem{Appell26}
Appell~P, {Kamp{\'e} de F{\'e}riet}~J. Fonctions hyperg{\'e}om{\'e}triques et
  hypersph{\'e}riques. Polynomes d'{H}ermite. Paris: Gauthier-Villars; 1926.

\bibitem{Bailey}
Bailey~WN. Generalized hypergeometric series. New York: Hafner; 1972.

\bibitem{Slater}
Slater~LJ. Generalized hypergeometric functions. Cambridge, London and New
  York: Cambridge University Press; 1966.

\bibitem{SriMan}
Srivastava~HM, Manocha~HL. A treatise on generating functions. John Wiley and
  Sons, New York, Chichester, Brisbane and Toronto: Halsted Press (Ellis
  Horwood Limited, Chichester)/Wiley; 1984.

\bibitem{ErdT}
Erd\'elyi~A, Magnus~W, Oberhettinger~F, Tricomi~FG. Higher transcendental
  functions. Vol.~1. New York: McGraw-Hill; 1953; "Horn's List" and
  "Convergence of the Series." В \S5.7.1 and \S5.7.2: 224--229.

\bibitem{SriKar}
Srivastava~HM, Karlsson~PW. Multiple Gaussian hypergeometric series. John Wiley
  and Sons, New York, Chichester, Brisbane and Toronto: Halsted Press (Ellis
  Horwood Limited, Chichester); 1985.

\bibitem{Sh07}
Shpot~MA. A massive {F}eynman integral and some reduction relations for
  {A}ppell functions. J Math Phys. 2007; 48(12): 123512 (pp. ~13).

\bibitem{DavDel98}
Davydychev~AI, Delbourgo~R. A geometrical angle on {F}eynman integrals. J Math Phys. 1998; 39(9): 4299--4334.

\bibitem{Tar08}
Tarasov~OV. New relationships between {F}eynman integrals. Phys Lett B. 2008; 670(1): 67--72.

\bibitem{BDS96}
Berends~FA, Davydychev~AI, Smirnov~VA. Small-threshold behaviour of two-loop
  self--energy diagrams: two-particle thresholds. Nucl Phys B. 1996; 478(3): 59--89.

\bibitem{BG72}
Bollini~CG, Giambiagi~JJ. Dimensional renormalization: The number of
  dimensions as a regularizing parameter. Il Nuovo Cimento B. 1972; 12(1): 20--26.

\bibitem{BD91}
Boos~\'E\'E, Davydychev~AI. A method of calculating massive {F}eynman integrals.
  Teor Mat Fiz. 1991; 89: 56--72; [Sov Phys Theor Math Phys. 1992; 89: 1052--1064].

\bibitem{Kre92}
Kreimer~D. One loop integrals revisited. Z Phys C: Particles and Fields.
  1992; 54(4): 667--671.

\bibitem{KSS82}
Kalinowski~MW, Sewerynski~M, Szymanowski~I. On some dimensional regularisation
  for Feynman massless integrals. J Phys A. 1982; 15(6): 1909--1929.

\bibitem{KT12}
Kniehl~BA, Tarasov~OV. Finding new relationships between hypergeometric
  functions by evaluating {F}eynman integrals. Nucl Phys B. 2012; 854(3): 841--852.

\bibitem{BS12}
Brychkov~{\relax Yu}A, Saad~N. Some formulae for the {A}ppell function $F_1(a, b, b'; c; w, z)$. Integral Transforms Spec
Funct. 2012; 23(11): 793--802.

\bibitem{SS15}
Shpot~MA, Srivastava~HM. The {C}lausenian hypergeometric function with unit argument and negative integral parameter differences.
Appl Math Comput. 2015; 259: 819 -- 827.

\bibitem{WIP}
Shpot~MA, Pog\'any~TK. work in progress.

\bibitem{Carl76}
Carlson~BC. Quadratic transformations of {A}ppell functions. SIAM Journal on
  Mathematical Analysis. 1976; 7(2): 291--304.

\bibitem{SR14}
Scarpello~GM, Ritelli~D. On computing some special values of multivariate
  hypergeometric functions. J Math Anal Appl. 2014; 420(2): 1693--1718.

\bibitem{HM06}
Huber~T, Ma{\^{\i}}tre~D. Hyp{E}xp, a {M}athematica package for expanding
  hypergeometric functions around integer-valued parameters. Comput Phys Commun. 2006; 175(2): 122--144.

\bibitem{PBM1}
Prudnikov~AP, Brychkov~{\relax Yu}A, Marichev~OI. Integrals and series.
  Elementary functions. Vol.~1. New York: Gordon and Breach; 1986.

\bibitem{PBM3}
Prudnikov~AP, Brychkov~{\relax Yu}A, Marichev~OI. Integrals and series. More special functions. Vol.~3.
  New York: Gordon and Breach; 1990.

\bibitem{AAR}
Andrews~GE, Askey~R, Roy~R. Special functions. In: Encyclopedia of mathematics
  and its applications. Vol.~71; Cambridge: Cambridge University Press; 1999.

\bibitem{Deshp74}
Deshpande~VL. Confluent expansions for functions of two variables. Math Comp. 1974; 28(2): 605--611.
\end{thebibliography}
\end{document}